\documentclass[11pt]{amsart}%reqno coloca a numeracao das equacoes a direita
\usepackage[left=1in,right=1in,bottom=0.5in,top=0.8in]{geometry}
\usepackage{amsfonts}% carregar fontes matemáticas.
\usepackage{amsmath, amssymb}% carregar novos comandos matemáticos.
\usepackage{graphicx} % inserir figuras no texto
\usepackage{epstopdf}
\usepackage[colorlinks, linkcolor=black, urlcolor=blue]{hyperref}
\usepackage{xcolor}
\usepackage{amsthm}%definir estilo dos ambientes(newtheorem).
\usepackage{listings}
\usepackage{longtable}
\usepackage{mathrsfs}
\usepackage[utf8]{inputenc}  % acentuação direta pelo teclado.
\usepackage[T1]{fontenc}
\usepackage{indentfirst} % identar o primeiro parágrafo.
\usepackage{enumerate} % enumeração automática personalizada
\usepackage{fancyhdr} % rodapé e cabeçalho de páginas
\usepackage{verbatim} %fazer comentários
\usepackage{float} % fixar figuras
\usepackage{mathtools}
\mathtoolsset{showonlyrefs}
\usepackage[numbers, sort]{natbib}

\newtheorem{teo}{Theorem}

\theoremstyle{definition} %coloca o título em negrito e o corpo normal

\newtheorem{lemma}[teo]{Lemma}

\newtheorem{example}{Example}

\newcommand{\B}{B^n}
\newcommand{\RR}{(\mathbb{R}^n)}

\DeclareMathOperator{\Conv}{Conv}
\DeclareMathOperator{\I}{I} %indicator function
\DeclareMathOperator{\dom}{dom}
\DeclareMathOperator{\interior}{int}
\DeclareMathOperator{\MA}{MA}
\DeclareMathOperator{\MAn}{MA,h}
\DeclareMathOperator{\Hess}{D^2}
\DeclareMathOperator{\epi}{epi}
\DeclareMathOperator{\CD}{Conv_{bd}}
\DeclareMathOperator{\lip}{Conv_{lip}}
\DeclareMathOperator{\lipd}{Conv^*_{lip}}
\newcommand{\finite}{(\mathbb{R}^n; \mathbb{R})}

\DeclareMathOperator{\supp}{supp}
\newcommand{\conc}{{\mbox{\rm Conc}([0,\infty))}}
\newcommand{\cvx}{{\mbox{\rm Cvx}([0,\infty))}}
\newcommand{\F}{\cF_{\Phi}(X)}
\newcommand{\seq}{\xrightarrow{\tau}}
\newcommand{\dseq}{\xrightarrow{\tau^*}}
\renewcommand{\d}{\,\mathrm{d}}
\newcommand{\oZ}{\operatorname{Z}}
\newcommand{\R}{{\mathbb R}}
\newcommand{\cK}{{\mathcal K}}
\DeclareMathOperator{\bd}{bd}
\newcommand{\oom}{C}

\newcommand{\fconvx}{{\mbox{\rm Conv}(\R^n)}} %convex on R^n
\newcommand{\fconvdl}{{\mbox{\rm Conv}_{{\rm ld}}(\R^n)}} %convex, compact domain, Lipschitz on domain
\newcommand{\fconvdld}{{\mbox{\rm Conv}^*_{{\rm ld}}(\R^n)}} %dual of convex, compact domain, Lipschitz on domain
\newcommand{\fconvf}{{\mbox{\rm Conv}(\R^n; \R)}}% finite convex functions on R^n
\newcommand{\cF}{{\mathcal F}}

\author{Fernanda Moreira Baêta  and Monika Ludwig}
\address{Institut für Diskrete Mathematik und Geometrie, Technische Universität Wien, Wiedner Hauptstr.\  8-10/1046, 1040 Wien, Austria}
\subjclass[2020]{26B25,  52A41, 26B15, 28A33, 28A25}
\title[]{On the Semicontinuity of Functionals on Function Spaces}

\begin{document}

\begin{abstract}
Results on the upper and lower semicontinuity of functionals defined on spaces of convex and more general functions are established. In particular, the following result is obtained.
Let  $\phi(v; \cdot)$ be the density of the absolutely continuous part of a Radon measure $\Phi(v; \cdot)$ 
associated to a function $v\colon X\rightarrow \mathbb{R}$ defined on the topological measure space $(X,\lambda)$. For concave $\zeta\colon [0, \infty)\rightarrow[0,\infty)$ with $\lim_{t\to 0} \zeta(t)=0$ and $\lim_{t\to\infty}\zeta(t)/t= 0$, it is shown that   the functional $v\mapsto\int_{X} \zeta(\phi(v;x))\d\lambda(x)$ depends upper   semicontinuously on $v$. Examples include %so-called 
functional affine surface areas for convex functions.
\end{abstract}
\maketitle

\section{Introduction}
 For a non-empty, compact, convex $K\subset \R^{n+1}$, define the affine surface area as
\begin{equation}\label{affine_surface_area}
     \int_{\bd K} \kappa(K,x)^{\frac{1}{n+2}}\d x,
\end{equation}
where $\kappa(K,x)$ is the generalized Gaussian curvature of $K$ at $x\in\bd K$, the boundary of $K$, and $\d x$ stands for integration with respect to the $n$-dimensional Hausdorff measure on the boundary of $K$. Let $\cK^{n+1}$ denote the space of convex bodies (non-empty, compact, convex sets) in $\R^{n+1}$ equipped with the topology induced by the Hausdorff metric. Since affine surface area vanishes on polytopes, it is not continuous on $\cK^{n+1}$. However, as shown by Lutwak \cite{Lutwak91}, affine surface area is an upper semicontinuous functional on $\cK^{n+1}$. The upper semicontinuity of a wider class of curvature functionals was established in \cite{Ludwig:curvature, Lutwak92}.

Let $D$ be a  bounded convex domain of $\mathbb{R}^n$
and  $v\in C^2(D)$ convex. The affine surface area of the epi-graph of $v$ is  
\begin{align}\label{ab}
  \int_{D} \det (\Hess v(x))^{\frac{1}{n+2}}\d x,
\end{align}
where $\Hess v(x)$ is the Hessian matrix of $v$ at $x$ and  $\d x$ stands for integration with respect to $n$-dimensional Lebesgue measure. 
By a classical theorem of Aleksandrov (see \cite[Theorem 6.4.1]{EvansGariepy}), 
a convex function $v$ is twice differentiable almost everywhere, and the integral in \eqref{ab} is well-defined (but possibly infinite).% for a general convex function $v\colon D\to\R$.

In [\citealp{Trudinger:Wang2000}, Lemma 6.4], Trudinger and Wang show that the functional defined in \eqref{ab} is upper semicontinuous for a given bounded convex domain $D$ with respect to locally uniform convergence, using that  $\det (\Hess v)$ is  the density of the absolutely continuous part of the Monge--Amp\`ere measure of $v$ and the weak$^*$-continuity of  Monge--Amp\`ere measures (see Section \ref{tools}). For further definitions of functional affine surface area, see, for example, \cite{AKSW, Caglar, LSW, 
STTW}.

\goodbreak
This paper aims to generalize this result to a specific class of extended real-valued convex functions  without restricting to a bounded set $D$. We also extend the result to more general functionals. The main motivation comes from the quest for classification results of valuations on convex functions. Such classification results have been recently established for continuous valuations on convex functions (see, for example,  \cite{Colesanti-Ludwig-Mussnig-4, Colesanti-Ludwig-Mussnig-5}). On convex bodies, they are classical for continuous valuations (see \cite[Chapter 6]{Schneider:CB2}) and have also been established for upper semicontinuous valuations in certain cases (see \cite{Ludwig:semi, Ludwig:affinelength, Ludwig:Reitzner, Ludwig:Reitzner2}).

Let $\fconvx$ be the space of 
extended real-valued functions 
$u\colon\R^n\to(-\infty, \infty]$ that are convex and lower semicontinuous and such that $u\not\equiv\infty$.
The (effective) {domain} of $u$ is $\dom (u)=\{x\in\R^n: u(x)<\infty\}$ and we write $\interior$ for interior.
We consider the subspace 
\begin{equation}
 \fconvdl=
\{u\in \fconvx\colon \dom (u) \mbox{ is compact},\, u \mbox{ is Lipschitz on $\interior \dom (u)$}\}.
\end{equation}
We use epi-convergence on $\fconvx$ and its subspaces, and we also use so-called $\tau$-convergence  on $\fconvdl$ (see Section \ref{tools}).

First, we state the following result, which proves in particular the $\tau$-upper semi\-continuity of the functional defined in \eqref{ab} on $\fconvdl$. Let $\conc$ be the set of concave functions $\zeta\colon [0,\infty)\to [0,\infty)$ such that  $\lim_{t\to 0}\zeta(t)=0$ and $\lim_{t\to\infty} \zeta(t)/t=0$.

\begin{teo}\label{teo1p}
For  $\zeta\in \conc$ and $u\in\fconvdl$,
\begin{align*}
    \oZ(u)=\int_{\dom(u)} \zeta(\det(\Hess u(x)))\d x
\end{align*}
is finite, and 
\begin{equation}
    \oZ(u)\ge \limsup_{k\to \infty} \oZ(u_k)
\end{equation}
for every sequence of functions $u_k$ in $\fconvdl$ with uniformly bounded Lipschitz constants on their domains
that epi-converges to $u\in \fconvdl$.
\end{teo}

\noindent
We remark that the restriction to sequences with uniformly bounded Lipschitz constants is necessary (see Example \ref{example}). 

We will prove Theorem \ref{teo1p} and its generalizations in the dual setting obtained by using the Legendre transform. In the dual setting, we  look at convex functions $v\colon \R^n\to \R$ whose Monge--Amp\`ere measures have compact supports and at sequences  $v_k: \R^n\to \R$ with uniformly bounded supports of their Monge--Amp\`ere measures. We establish the results in greater generality for densities of the absolutely continuous part of Radon measures associated to functions on topological spaces in Theorem \ref{main} and also provide results on lower semicontinuity in Theorem~\ref{lower}.

\section{Preliminaries}\label{tools}

We collect results on Radon measures and on convex functions.

\subsection{Radon Measures on Metric Spaces}

Let $(X, \lambda)$ be a metric space with a (non-negative) $\sigma$-finite Radon measure $\lambda$. We will consider $\sigma$-finite Radon measures depending on certain functions on $X$  and use the Radon--Nikodým and Lebesgue decomposition theorems  (see, for example,  [\citealp{EvansGariepy}, Theorems  1.6.2 and 1.6.3]).

\goodbreak
Let $\cF(X)$ be a space of functions $v\colon X\to \R$ equipped with a notion of convergence.
For $v\in \cF(X)$, let $\Phi(v;\cdot)$ be a Radon measure depending on $v$ defined on Borel sets of $X$  such that we have weak$^*$-star convergence of $\Phi(v_k;\cdot)$ to $\Phi(v;\cdot)$ as the sequence $v_k$ converges to $v\in\cF(X)$. Here, we say that $\Phi(v_k; \cdot)$ is {weak$^*$-convergent} to $\Phi(v; \cdot)$ if for any sequence $v_k$ converging in $\cF(X)$ to $v$, we have
\begin{align*}
    \lim_{k\rightarrow \infty} \int_X \beta(x)\d\Phi(v_k;x)= \int_X \beta(x)\d\Phi(v;x)
\end{align*}
for every function $\beta\in C_c(X)$, the set of continuous functions with compact support on $X$. We need the following simple result. For a sequence of functions $v_k\in \cF(X)$ converging to $v$, the weak$^*$-convergence of $\Phi(v_k; \cdot)$ to $\Phi(v; \cdot)$ implies  that
\begin{align}\label{weak-star}
  \limsup_{k\rightarrow \infty} \Phi(v_k;\oom)\leq \Phi(v;\oom),  
\end{align}
for every compact set $\oom\subset X$ (cf.\ [\citealp{Ash}, Theorem 4.5.1]). 

By the Lebesgue decomposition theorem, the measure $\Phi(v; \cdot)$ can be decomposed into measures absolutely continuous and singular with respect to $\lambda$ on $X$, say $\Phi(v; \cdot) = \Phi^a(v; \cdot)+\Phi^s(v; \cdot)$. By the Radon--Nikodým theorem, we have for the absolutely continuous part
\begin{align*}
\Phi^a(v; A)= \int_A \phi(v; x) \d\lambda(x), 
\end{align*}
where $A\subseteq X$ is a $\lambda$-measurable set and $\phi(v;\cdot)$ is $\lambda$-measurable, while  for the singular part, there is a set $\oom_0\subset X$ such that $\lambda(\oom_0)=0$ and 
%\begin{align*}
$\Phi^s(v; A\setminus \oom_0)= 0$,
%\end{align*}
for every $\lambda$-measurable set $A\subseteq X$.

\subsection{Convex Functions}
Let $\Conv\finite$ denote the space of finite-valued, convex functions on $\mathbb{R}^n$, and we also consider the space of Lipschitz functions
\begin{align*}
    \lip\finite = \{ v \in \Conv\finite : v \ \text{is Lipschitz}\}
\end{align*}
and the space
$$\CD(\mathbb{R}^n)=\{u\in \Conv(\mathbb{R}^n): \dom (u)  \text{ is bounded}\}.$$ 
We say that a sequence $u_k\in \Conv(\mathbb{R}^n)$ or any of its subspaces {epi-converges} to $u\in\Conv(\mathbb{R}^n)$ if 
for every sequence $x_k$ that converges to $x$, we have 
$$u(x)\leq \liminf_{k\rightarrow \infty} u_k(x_k),$$ and there exists a sequence $x_k$ that converges to $x$ such that 
$$u(x)= \lim_{k\rightarrow \infty}u_k(x_k).$$
Epi-convergence is directly related to uniform convergence since $u_k$ epi-converges to $u$ if and only if $u_k$  converges uniformly to $u$  on every compact set that does not
contain a boundary point of $\dom (u)$ (see [\citealp{RockafellarWets}, Theorem 7.17]). In this paper, we also use the following  notion of convergence: For $u,u_k\in\fconvdl$ with $k\in\mathbb{N}$, 
we say that  $u_k\seq u$ if
\begin{enumerate}
\item[(i)] $u_k$ epi-converges to $u$;
\item[(ii)] the Lipschitz constants of $u_k$ are uniformly bounded on $\interior \dom (u)$
by some constant that does not depend on $k$.
\end{enumerate}
Condition (ii)  for the space of real-valued Lipschitz continuous maps defined on the unit sphere $S^{n-1}$
has been considered  in  \cite{Colesanti-Pagnini-Tradacete-Villanueva-2020, Colesanti-Pagnini-Tradacete-Villanueva-2021}  with the additional requirement of uniform convergence on the unit sphere. 
\goodbreak

\begin{example}\label{example}
The assumption of uniformly bounded Lipschitz constants is necessary for Theorem~\ref{teo1p} to hold. For $x\in\R^n$, consider
$$u_k(x)=k\sum_{i=1}^n\langle x, x\rangle +\I_{B^n}(x),$$ 
where $B^n$ is the $n$-dimensional  Euclidean unit ball,  $\langle\cdot,\cdot\rangle$ is the inner product on $\R^n$, and $\I_{B^n}(x)=0$ if $x\in B^n$ and $\infty$ otherwise.  Then $u_k$ epi-converges to the  function  $\I_{\{0\}}$, which is $\infty$ for every  nonzero $x$. This shows that Theorem \ref{teo1p} fails if we replace $\tau$-convergence by epi-convergence since $\oZ$ is, in general, not upper semicontinuous with respect to epi-convergence. 
\end{example}

For $u\in \fconvx$,
let  $\partial u(x)$ be the {subdifferential} of $u$ at $x\in \dom (u)$, 
$$\partial u(x)=\{y\in\R^n: u(z)\geq u(x)+\langle y, z-x\rangle  \text{ for all } z\in \R^n\},$$
and set $\partial u(x)= \emptyset$ for $x\not\in \dom (u)$.
Given a subset $A\subset \R^n$,  the image of $A$ through the subdifferential of $u$ is defined  as
$$\partial u(A)= \bigcup\nolimits_{x\in A}\partial u(x).$$
For $v\in\fconvf$, the Monge--Ampère measure is given by
\begin{align}\label{MAdef}
 \MA(v; B)= V_n(\partial v(B))  
\end{align}
for every Borel set $B\subseteq \R^n$, where $V_n$ is the $n$-dimensional volume. It is a Radon measure on $\R^n$. Let 
\begin{equation}
\Conv_{\MA}(\R^n;\R)=\{ v\in \Conv\finite: \supp (\MA(v;\cdot)) \text{ is compact}\},
\end{equation}
where $\supp$ stands for support.

Let $K\subset\R^{n+1}$ be a non-empty, closed, convex set with support function $h_K: \R^{n+1}\to (-\infty, \infty]$, that is,
$$h_{K}(x, x_{n+1})=\sup\big\{\langle (x, x_{n+1}), (y,y_{n+1})\rangle\colon  (y,y_{n+1})\in K\big\} $$
for $x\in \R^{n}$ and $x_{n+1}\in\R$, where the inner product $\langle \cdot,\cdot\rangle$ is taken in $\R^{n+1}$. 
Note that if $K\in \cK^{n+1}$, we have $h_{K}(\cdot, -1)\in\Conv\finite$.
Let
\begin{equation}  
    \Conv_{\MAn}\finite=
    \{h_{K}(\cdot, -1)\in\Conv_{\MA}\finite: K\in\cK^{n+1}\}.
\end{equation}
Note that the inclusion $\Conv_{\MAn}(\mathbb{R}^n;\mathbb{R})\subset\Conv_{\MA}(\mathbb{R}^n;\mathbb{R})$  is strict.   Functions $ v\in\Conv_{\MAn}(\mathbb{R}^n;\mathbb{R})$ may have $\dom v^*$ not full-dimensional but always bounded,  whereas in $\Conv_{\MA}(\mathbb{R}^n;\mathbb{R})$ the domain of $v^*$ can be unbounded,  which only occurs when it is not full-dimensional.

We equip  $\Conv_{\MA}\finite$ with the following notion of convergence: we say that a sequence $v_k$ in $\Conv_{\MA}\finite$ is $\tau^*$-convergent to  $v\in \Conv_{\MA}\finite$ if
\begin{enumerate}
\item[(i)] $v_k$ epi-converges to $v$;
\item[(ii)] there exists a compact set in $\mathbb{R}^n$ containing the supports of $\MA(u; \cdot)$ and $\MA(u_k; \cdot)$ for all $k\in\mathbb{N}$.
\end{enumerate}
In this case we write $v_k \dseq v$. The notation will be justified in Lemma \ref{duality}.
\goodbreak
\section{The Dual Space}

For  $u\in \Conv(\mathbb{R}^n)$, the {Legendre transform} or convex conjugate of $u$ is the function $u^*: \R^n\to (-\infty, \infty]$ given as
\begin{align}
    u^*(y)=\sup \{\langle x,y \rangle - u(x)\colon  x\in\mathbb{R}^n\}
    \end{align}
for $y \in\mathbb{R}^n$.
If $u\in \CD(\mathbb{R}^n)$, then $u^*\in \Conv\finite$ (see [\citealp{RockafellarWets}, Theorem 11.8],  and $u^{**}=u$ (see [\citealp{Schneider:CB2}, Theorem 1.6.13]). Note that convex conjugation is continuous w.r.t.\ epi-convergence,
\begin{equation}\label{Wijsman}
   u_k \text{ is epi-convergent to } u\,\,\Leftrightarrow\,\, u^*_k \text{ is epi-convergent to } u^*
\end{equation}
for $u_k, u\in \fconvx$ (see \cite[Theorem 11.43]{RockafellarWets}).

The epigraph of $u\in\fconvx$ is the non-empty, closed, convex set
\begin{equation}
    \epi(u)=\{(x, x_{n+1})\in\R^{n+1}: x_{n+1} \ge u(x)\}.
\end{equation} 
Hence,
\begin{equation}\label{minus}
    u^*(y)=\sup\{\langle (y, -1),(x,u(x))\rangle\colon  x\in\mathbb{R}^n\} = h_{\epi(u)}(y,-1).
\end{equation}
If $\dom u$ is compact, then there is $K\in \cK^{n+1}$ such that $h_K(\cdot,-1)=h_{\epi(u)}(\cdot,-1)$,
which can be chosen as the intersection of $\epi(u)$ and $\{x_{n+1}\le \max_{x\in \dom (u)} u(x)\}$. Hence, the Legendre transform of a function in $\fconvx$ with compact domain is of the form
\begin{equation}\label{cd}
h_K(\cdot,-1)\, \text{ with } \,K\in\cK^{n+1}.
\end{equation}
For the space of Legendre transforms of functions from $\fconvdl$, we write
\begin{equation}\label{conj}
    \fconvdld=\{u^*:     u\in\fconvdl\}
\end{equation}
and use corresponding notation for further spaces. By [\citealp{Rockafellar-1997}, Corollary 13.3.3],
\begin{align}\label{LipBd}
 \lipd\finite= \CD\RR.   
\end{align}
Using this in its dual form  and \eqref{cd}, we obtain 
\begin{equation}\label{MALip}
\Conv_{\MAn}\finite\subset\lip\finite.
\end{equation}

By [\citealp{Rockafellar-1997}, Corollary 23.5.1], we have the following relationship
\begin{align}\label{sub}
    p\in \partial u(x) \quad  \Leftrightarrow \quad x \in \partial u(p)
\end{align}
for $u\in \Conv(\mathbb{R}^n)$. If  both $u$ and $u^*$ are differentiable, then we can rewrite  \eqref{sub} as
$\nabla u^*(\nabla u(p))=p$.
For $u$ and  $u^*$ twice differentiable and $\det ( \Hess u(p))>0$, differentiating  gives 
\begin{align*}
  \det( \Hess u^*(\nabla u(p)) )=\dfrac1 {\det ( \Hess u(p))}.
\end{align*}
Thus, if $u\in \fconvdl$ is such that $\det( \Hess u(x))>0$ for all $x\in \dom (u)$  and $\zeta\in\conc$, we have
\begin{align}
    \nonumber \int_{\dom (u)} \zeta (\det (\Hess u(x)))\d x &= \int_{\dom (u)}\zeta\left(\dfrac{1}{\det (\Hess u^*(\nabla u(x)))}\right)\d x\\
    \label{gamma} & = \int_{\mathbb{R}^n} \zeta\left(\dfrac{1}{\det (\Hess u^*(y))}\right) \det (\Hess u^*(y))\d y
    \\ &=  \int_{\mathbb{R}^n} \tilde{\zeta} (\det (\Hess u^*(x)))\d x,
\end{align}
where $\tilde{\zeta}(t)= \zeta(1/t)t$ for $t\geq 0$. Note that we also have $\tilde{\zeta}\in\conc$. Using \eqref{sub}, we observe that the last equality holds for every $u\in \fconvdl$. Indeed,  let $u\in \fconvdl$ and let $U$ be the  set of points $x\in\dom (u)$ such that $\det (\Hess u(x))=0$. By assumption, $U$ also contains all points where $u$ is not twice differentiable, 
since we set $\det (\Hess u(x)) = 0$ at such points. If
$U_1\subset U$ is the set of points where $u$ is not twice differentiable, then by Aleksandrov's theorem, $V_n(U_1)=0$. If $D= \partial u(U_1)$, we get by \eqref{MAdef} that $\MA(u^*; D)= V_n(U_1)=0$, which implies  $\MA^a(u^*; D)=0$. Now, if there exists a convex set $U_2\subset U$ such that $\det (\Hess u(x))=0$ for all $x\in U_2$ and $V_n(U_2)>0$, then there exists a vector $p\in\mathbb{R}^n$ such that $p\in\partial u(U_2)$ and by \eqref{sub} we have $U_2\subset \partial u^*(p)$, that is, $u^*$ is not twice differentiable at $p$. Using \eqref{gamma}, we conclude that 
\begin{align}
 \int\limits_{\dom (u)} \zeta (\det (\Hess u(x)))\d x 
 &=  \int\limits_{U} \zeta (\det (\Hess u(x)))\d x
 + \int_{\dom (u)\setminus U}\zeta (\det (\Hess u(x)))\d x\nonumber\\
 & =   \int_{\mathbb{R}^n}\tilde{\zeta}(\det (\Hess u^*(x)))\d x.\label{conjugate}
\end{align}

Let $u_k,u\in \fconvdl$. 
We aim to show that if the functions $u_k,u\in\fconvdl$ and $u_k \seq u$, then the functions $u_k^*, u^*\in \Conv_{\MAn}\finite$ and  $u_k^* \dseq u^*$.  Additionally, we will prove the reverse implication. To do this, we will use the following result (see, for example, [\citealp{beck}, Theorem 3.61]).

\begin{lemma}\label{est}
A function $u\in \Conv\RR$ has a  Lipschitz constant  $L$    in the interior of its domain  if and only if $\Vert p\Vert\leq L$ for every $p\in\partial u(x)$, where $x\in \interior \dom (u)$.    
\end{lemma}

Here, $\|\cdot\|$ denotes the Euclidean norm on $\R^n$. We now prove our main lemma.

\begin{lemma}\label{duality}
A sequence $u_k$ in $\fconvdl$ is $\tau$-convergent to $u \in \fconvdl$ if and only if $u_k^*$ and $u^*$ belong to 
$\Conv_{\MAn}\finite$ and $u_k^*$ is $\tau^*$-convergent to  $u^*$.
\end{lemma}

\begin{proof}
First, we show that for $\Conv^*_{\MAn}\finite$, defined similar to \eqref{conj}, 
\begin{equation}\label{MaLip}
    \Conv^*_{\MAn}\finite\subset \fconvdl.
\end{equation}
Let $u\in \Conv^*_{\MAn}\finite$. By \eqref{MALip} and \eqref{LipBd}, we see that $u\in \CD\RR$. Since $u$ is the Legendre transform of $h_K(\cdot,-1)$ for some $K\in \cK^{n+1}$, the domain of $u$ is compact. If 
$\interior\dom (u)$ is empty, then we do not need the Lipschitz condition on $u$ and hence  $u \in \fconvdl$.  
For $\interior\dom(u)$ non-empty, assume, for the sake of contradiction, that
$u$ is not Lipschitz on $\interior\dom(u)$. By Lemma ~\ref{est},   for every constant $M>0$, there exists $z\in\partial u(\interior \dom (u) )$   such that $\|z\|> M$. Using the continuity and  convexity of $u$, we deduce  that for every sufficiently small neighborhood  $D\subset\mathbb{R}^n$  of $z$, there exists a set $U\subset\interior \dom (u)$ with positive Lebesgue measure such that
$D\subset \partial u(U)$.  By \eqref{sub},
$$D\subset \partial u(U) \quad \Leftrightarrow \quad    U\subset \partial u^*(D),$$
which implies that $\MA(u^*;D)>0$. Since this holds for every sufficiently small neighborhood of $z$, we obtain that $z\in\supp (\MA(u^*; \cdot))$. Therefore, we have $\supp (\MA(u^*; \cdot))\not\subset M\B$. Since $M>0$
was arbitrary, this leads to a contradiction to % with the fact that 
$u^* \in \Conv_{\MAn}\finite$.
Hence, we conclude that $u$ is Lipschitz on $\interior\dom (u)$ and, thus, $u\in \fconvdl$. 

Now, assume that $u_k^*, u^*\in \Conv_{\MAn}\finite$ for every $k\in\mathbb{N}$ and that $u_k^*\dseq u^*$ as $k\to \infty$. From \eqref{MaLip}, it follows that $u_k,u\in  \fconvdl$ for every $k\in\mathbb{N}$. Let $MB^n$ be the ball containing the support  of the Monge--Ampère measures of $u_k^*$ and $u^*$ for every  $k\in\mathbb{N}$. We have seen that
\begin{align*}
    \partial u_k(\interior \dom (u_k) )\not\subset M\B \quad \Rightarrow \quad \supp (\MA(u_k^*; \cdot))\not\subset M\B.
\end{align*}
Consequently, we have $ \partial u_k(\interior \dom (u_k)),  \partial u(\interior \dom (u) )\subseteq M\B$.  By Lemma~\ref{est}, this is equivalent to saying that the Lipschitz constants of $u_k$ and $u$ are bounded by $M$. Using \eqref{Wijsman}, we conclude that $u_k\seq u$.

For the converse, assume that $u\in \fconvdl$ with Lipschitz constant $L_u$.
By \cite[p.\ 227]{Rockafellar-1997}, we have
\begin{equation}\label{monge}
    \MA(u^*;\R^n) = V_n(\partial u^*(\R^n))= V_n(\dom (u)).
\end{equation}
If $V_n(\dom (u))=0$,  this implies that $\supp (\MA(u^*;\cdot))$ is compact.
For the case $V_n(\dom (u))>0$,
set
    $D= \partial u(\interior \dom (u))$,
and note that, by \cite[p.\ 227]{Rockafellar-1997}, we have
  $\MA(u^*; D)= V_n(\dom (u)).$
Hence, it follows from \eqref{monge} that $\MA(u^*;\cdot)$ is supported on $D$. By Lemma \ref{est}, we have $D \subseteq L_u \B$, which implies that
\begin{equation}\label{1000}
\supp (\MA(u^*;\cdot))\subseteq L_u\B.
\end{equation}
Combined with the fact that $u$ has compact domain and \eqref{cd}, this implies that $u^*\in \Conv_{\MAn}\finite$. Hence, 
$\fconvdld=\Conv_{\MAn}\finite$.

To conclude the proof, it  remains to show that if $u_k,u\in \fconvdl$ for every $k\in\mathbb{N}$ and $u_k\seq u$ as $k\to \infty$, then $u_k^*\dseq u^*$. Let $M>0$ be  such that the Lipschitz constants of $u_k$ and $u$ are bounded by $M$. By \eqref{1000}, we know that
$$\supp (\MA(u^*;\cdot))\subset L_{u}\B \qquad \mbox{and} \qquad  \supp (\MA(u_k^*;\cdot))\subset L_{u_k}\B,$$
where $L_u$ and $L_{u_k}$ are the Lipschitz constants of $u$ and $u_k$, respectively. Since 
$L_{u_k}, L_u\leq M$
for every $k\in\mathbb{N}$, we conclude that 
$$\supp (\MA(u^*;\cdot)), \, \supp (\MA(u_k^*;\cdot))\subset M\B,$$
and therefore, by \eqref{Wijsman}, we have $u_k^*\dseq u^*$.
This completes the proof.
\end{proof}

\section{Upper and Lower Semicontinuity}\label{upper}

Let $X$ be a topological space with $\sigma$-finite Radon measure $\lambda$. Let $\cF(X)$ be a family of real-valued functions on $X$ equipped with a notion of convergence, and $\Phi(v; \cdot)$  a family of $\sigma$-finite Radon measures depending weak$^*$-continuously on the functions $v\in \cF(X)$.
Let  $\cF_{\Phi}(X)$ denote the set of functions $v\in \cF(X)$ such that the support of $\Phi(v;\cdot)$ is compact,  and let $\phi(v;\cdot)$ denote    the  density of the absolutely continuous part of $\Phi(v;\cdot)$ for $v\in \cF(X)$.  We equip $X\times \R$ with the product topology. The following result is a functional version of Theorem 2 in \cite{Ludwig:curvature}.
\goodbreak

\begin{teo}\label{main}
For  $\zeta\in \conc$ and non-negative $\omega\in C(X\times\mathbb{R})$,
\begin{align*}
    \oZ(v)=\int_{X} \zeta(\phi(v;x))\,\omega(x,v(x))\d\lambda(x),
\end{align*}\
is finite for every $v\in \F\cap C(X)$, and 
\begin{equation}
   \oZ(v)\ge \limsup_{k\to\infty} \oZ(v_k)
\end{equation}
for every sequence of functions $v_k$ in $\F\cap C(X)$ with uniformly bounded supports $\supp (\Phi(v_k;\cdot))$
that converges to $v\in\F\cap C(X)$.
\end{teo}
\goodbreak

Hence we can say that $\oZ$ is $\tau$-upper semicontinuous on $\F\cap C(X)$.
In the proof of Theorem  \ref{main} and Theorem \ref{lower}, we will use  Jensen's inequality, which guarantees that for a concave function $\zeta\colon [0,\infty)\rightarrow \mathbb{R}$ and an integrable function $v\colon C\rightarrow \mathbb{R}$, where $C$ is a compact subset of $X$ with positive $\lambda$-measure, 
\begin{align}\label{jensen}
   \dfrac{1}{\lambda(C)}\int_C \zeta(v(x)) \d\lambda(x) \leq \zeta\left( \dfrac{1}{\lambda(C)}\int_C v(x)\d\lambda(x)\right).
\end{align}
(cf.\ [\citealp{Rudin}, Theorem 3.3]). 
If $\zeta$ is convex, the inequality is reversed. 

Note that $\zeta\in\conc$ is continuous on $[0,\infty)$ and that $\zeta(0)=0$. Since $\zeta$ is concave and non-negative on $[0,\infty)$, the function $\zeta$ is non-decreasing. Using these properties, we get $\zeta(r t)\geq r \zeta(t)$
for every $t>0$ and $0<r<1$. In particular, if we take $s=r t<t$ we obtain 
$$\frac{\zeta(s)}{s}\geq \frac{\zeta(t)}{t},$$
and this means  that $t\mapsto{\zeta(t)}/{t}$ is non-increasing.

\begin{proof}[Proof of Theorem \ref{main}]
    
Let $\varepsilon>0$ be given, $v\in \F\cap C(X)$, and $\oom\subset X$ compact. Let $\oom_0$ be the set where the singular part of $\Phi(v;\cdot)$ is concentrated. Since $\lambda(\oom_0)=0$ and $\phi(v;\cdot)$  is measurable $\lambda$-a.e.\ on $\oom$, we can choose by Lusin's theorem (see, for example, [\citealp{Ash}, Corollaries 4.3.17]) a compact set $\oom_\varepsilon\subset \oom$ where $\phi(v;\cdot)$ is continuous as a function restricted to $\oom_\varepsilon$ such that
$\oom_\varepsilon \cap \oom_0= \emptyset$
and 
\begin{align}\label{equ5.12}
\lambda(\oom \setminus \oom_\varepsilon)\leq \varepsilon.
\end{align}

Let $v_k$ be a sequence in $\F\cap C(X)$ converging to $v$. First, we show that 
\begin{equation}\label{equ5.13}
    \limsup_{k\rightarrow\infty} \int_{\oom_\varepsilon}\zeta(\phi(v_k;x))\,\omega(x,v_k(x))\d\lambda(x)
    \leq \int_{\oom_\varepsilon} \zeta(\phi(v;x))\,\omega(x,v(x))\d\lambda(x)
\end{equation}
holds. Set
$a=\inf\{\phi(v;x)\colon x\in \oom_\varepsilon\}$ and  $b=\sup\{\phi(v;x)\colon x\in \oom_\varepsilon\}$.
Note that $b<\infty$ since $\phi(v;\cdot)$ is continuous on the compact set $\oom_{\varepsilon}$. Hence, $\zeta$ is uniformly continuous on $[a,b]$.  For $\eta>0$ given,  there exists a number $\delta>0$ such that
\begin{align}\label{equ5.14}
    |\zeta(s)-\zeta(t)|\leq \eta
\end{align}
whenever $s,t\in [a,b]$ are such that $|s-t|\leq \delta$. 

Since $\lambda(\{x\in\oom_\varepsilon: \phi(v;x)=t\})>0$ holds only for countably many $t$, we can choose  a subdivision $a=t_1\leq t_2\leq \cdots \leq t_{m+1}=b$ of $[a,b]$, such that 
\begin{align}\label{equ5.15}
\max_{i=1,\dots, m}(t_{i+1}-t_i)\leq \delta
\end{align}
and such that 
\begin{align}\label{zeroo}
\lambda(\{x\in\oom_{\varepsilon}: \phi(v;x)=t_i\})=0
\end{align}
for $i=2,\dots, m$.  Setting
$\oom_i=\{x\in \oom_\varepsilon: t_i\leq \phi(v;x)\leq t_{i+1}\}$, so that $\oom_{\varepsilon}=\cup_{i=1}^m  \oom_i$, and 
using the linearity of the integral, that $\zeta$ is non-decreasing, and \eqref{zeroo}, we get
\begin{align}\label{equ5.16}
\int_{\oom_\varepsilon} \zeta(\phi(v; x))\,\omega(x,v(x))\d\lambda(x) & \geq \sum_{i=1}^m \zeta(t_i)\int_{\oom_i} \omega(x,v(x))\d\lambda(x).
\end{align}
Consider $J_k\subseteq\{1, \dots, m\}$, depending on $k$ such that $\int_{\oom_i} \omega(x,v_k(x))\d\lambda(x)>0$ precisely if $i\in J_k$. By  Jensen's inequality and the concavity of $\zeta$, %the following inequality
\begin{multline}
\int_{\oom_i}\zeta(\phi(v_k;x))\,\omega(x,v_k(x))\d\lambda(x)\\
\leq \zeta\Big(\dfrac{\int_{\oom_i}\omega(x,v_k(x))\,\phi(v_k;x)\d\lambda(x)}{\int_{\oom_i}\omega(x,v_k(x))\d\lambda(x)}\Big)\int_{\oom_i}\omega(x,v_k(x))\d\lambda(x)
\end{multline}
holds for each $i\in J_k$. Hence,
\begin{multline}
 \int_{\oom_\varepsilon}\zeta(\phi(v_k; x))\,\omega(x,v_k(x))\d\lambda(x) \\ 
 \leq \sum_{i\in J_k} \zeta\Big(\dfrac{\int_{\oom_i}\omega(x,v_k(x))\,\phi(v_k;x)\d \lambda(x)}{\int_{\oom_i}\omega(x,v_k(x))\d\lambda(x)}\Big)\int_{\oom_i}\omega(x,v_k(x))\d\lambda(x).
\end{multline}
Since $\phi(v; \cdot)$ is continuous on the compact set $\oom_\varepsilon$, the   sets $\oom_i$ are compact for $i=1,\dots, m$. Hence, the functions $\omega(\cdot,v_k(\cdot))$ and $\omega(\cdot,v(\cdot))$ are uniformly continuous on $\oom_i$. Then, for any $\tilde{\varepsilon} > 0$, there exists $k_0$ such that for all $k \geq k_0$, 
\begin{align*}
\int_{\oom_i} \omega(x, v_k(x)) \d\Phi(v_k; x) 
&\leq  \int_{\oom_i} \omega(x, v(x)) \d\Phi(v_k; x) + \tilde{\varepsilon}\,  \Phi(v_k; \oom_i).
\end{align*}
Using that $\Phi(v_k; \cdot)$ weak$^*$-converges to $\Phi(v; \cdot)$, and in particular \eqref{weak-star}, we obtain 
\begin{align*}
\limsup_{k\to\infty} \int_{\oom_i} \omega(x, v_k(x)) \d\Phi(v_k; x) 
&\leq \int_{\oom_i} \omega(x, v(x)) \d\Phi(v; x) + \tilde{\varepsilon} \, \Phi(v; \oom_i).
\end{align*}
Since $\tilde{\varepsilon} > 0$ was arbitrary, it follows that
\[
\limsup_{k\to\infty}\int_{\oom_i}\omega(x,v_k(x))\d\Phi(v_k;x)\leq \int_{\oom_i}\omega(x,v(x))\d\Phi(v;x).
\]
%Together with \eqref{weak-star}, this gives
%\begin{align*}
%\limsup_{k\to\infty}\int_{\oom_i}\omega(x,v_k(x))\d\Phi(v_k;x)\leq \int_{\oom_i}\omega(x,v(x))\d\Phi(v;x).
%\end{align*}
Hence, 
\begin{align}
\limsup_{k\to\infty}& \int_{\oom_\varepsilon} \zeta(\phi(v_k;x))\,\omega(x,v_k(x))\d\lambda(x) \\
&\leq \limsup_{k\to\infty}
\sum_{i\in J_k} \zeta\Big(\dfrac{\int_{\oom_i}\phi(v_k;x)\,\omega(x,v_k(x))\d\lambda(x)}{\int_{\oom_i}\omega(x,v_k(x))\d\lambda(x)}\Big)\int_{\oom_i}\omega(x,v_k(x))\d\lambda(x)\\
\label{eqqu9}
&\leq \sum_{i=1}^m \zeta(t_{i+1})\int_{\oom_i}\omega(x,v(x))\d\lambda(x)\\
&= \sum_{i=1}^m (\zeta(t_i)+(\zeta(t_{i+1})-\zeta(t_i)))\int_{\oom_i}\omega(x,v(x))\d\lambda(x).
\end{align}
Note that we used the definition of $\oom_i$ and the monotonicity of $\zeta$ in the second inequality. Now using  \eqref{eqqu9}, \eqref{equ5.16}, \eqref{equ5.15}, and \eqref{equ5.14}, we conclude that
\begin{multline}
 \limsup_{k\to\infty} \int_{\oom_\varepsilon} \zeta(\phi(v_k; x))\,\omega(x,v_k(x))\d\lambda(x) \\
  \leq \int_{\oom_\varepsilon} \zeta( \phi(v;x))\,\omega(x,v(x))\d\lambda(x) + \eta \int_{\oom_\varepsilon}\omega(x,v(x))\d\lambda(x).
\end{multline}
Since $\eta>0$ is arbitrary, this proves \eqref{equ5.13}.
\goodbreak

The second step is  to show that
\begin{equation}\label{equ5.17}
\limsup_{k\to\infty}\int_{\oom}\zeta(\phi(v_k;x))\,\omega(x,v_k(x))\d\lambda(x) \\
\leq \int_{\oom}\zeta( \phi(v; x))\,\omega(x,v(x))\d\lambda(x). 
\end{equation}
Since $\zeta$ is non-decreasing and $\zeta(t)/t$ is non-increasing, we see that for every $t>0$,
\begin{align*}
\int_{\oom\setminus \oom_\varepsilon}&\zeta(\phi(v_k;x))\,\omega(x,v_k(x))\d\lambda(x)  \\
& =  \int_{\{x\in \oom\setminus(\oom_0\cup \oom_\varepsilon): \ \phi(v_k;x)\leq t\}}\zeta(\phi(v_k;x))\,\omega(x,v_k(x))\d\lambda(x)\\
& \qquad +\int_{\{x\in \oom\setminus (\oom_0\cup\oom_\varepsilon): \ \phi(v_k;x)> t\}}\zeta(\phi(v_k;x))\,\omega(x,v_k(x))\d\lambda(x) \\
&\leq \zeta(t)\int_{\oom\setminus \oom_\varepsilon}\omega(x,v_k(x))\d\lambda(x)+ \dfrac{\zeta(t)}{t}\int_{\oom\setminus \oom_\varepsilon}\omega(x,v_k(x))\, \phi(v,x)\d\lambda(x),
\end{align*}
where we use that $\lambda(\oom_0)=0$.
This implies, combined with  \eqref{equ5.13}  that  
\begin{multline}\limsup_{k\rightarrow \infty}\int_{\oom}\zeta(\phi(v_k;x))\,\omega(x,v_k(x))\d\lambda(x)
\leq \int_{\oom}\zeta(\phi(v;x))\,\omega(x,v(x))\d\lambda(x)\\
\,\,\,\,\,\,+\zeta(t)\!\int_{\oom\setminus \oom_\varepsilon}\!\!\,\omega(x,v(x))\d\lambda(x)
 +\frac{\zeta(t)}{t}\!\int_{\oom}\!\omega(x,v(x))\d\Phi(v, x)
\end{multline}
for every $t>0$. Since $\omega(\cdot,v(\cdot))$ is a non-negative continuous function and $\oom$ is compact, we have  $0\leq\max_{x\in\oom}\omega(x,v(x))<\infty$.
By this and \eqref{equ5.12}, we obtain
\begin{multline}
    \limsup_{k\rightarrow \infty}\int_{\oom}\zeta(\phi(v_k;x))\,\omega(x,v_k(x))\d\lambda(x)
\leq \\
\int_{\oom}\!\zeta(\phi(v;x))\,\omega(x,v(x))\d\lambda(x)
+\big(\varepsilon\,\zeta(t)
+\frac{\zeta(t)}{t}
\Phi(v; \oom)
\big)\max_{x\in\oom}\omega(x,v(x)).
\end{multline}
Since $\varepsilon>0$ is arbitrary and since $t$ does not depend on $\varepsilon$,  for every $t>0$
\begin{equation}
\label{equ5.18}
\begin{split}
\limsup_{k\rightarrow \infty}&\int_{\oom}\zeta(\phi(v_k;x))\,\omega(x,v_k(x))\d\lambda(x) \\
&\hskip -12pt\leq \int_{\oom}\zeta(\phi(v;x))\,\omega(x,v(x))\d\lambda(x)+\frac{\zeta(t)}{t}
\Phi(v; \oom)
\max_{x\in\oom}\omega(x,v(x)).   
\end{split}
\end{equation}
Using that $\Phi(v;\cdot)$  is a Radon measure  
and that $\lim_{t\rightarrow\infty}\zeta(t)/t=0$, we can make $\zeta(t)/t$ arbitrarily small by choosing $t$ suitably large. Therefore \eqref{equ5.18} proves \eqref{equ5.17}.

To finish, if $\phi(v; \cdot)$ is supported on the compact set $\oom_v$ with $\lambda(\oom_v)>0$ and $\int_{\oom_v}\omega(x,v(x))\d\lambda(x)>0$, then using  that $\zeta(0)=0$, the monotonicity of $\zeta$ and that there exists a compact set $\tilde C\subset X$ such that $\oom_{v_k},\oom_v\subseteq \tilde C$ for all $k\in \mathbb{N}$, where $\oom_{v_k}$ is the support of $\phi(v_k;\cdot)$, we get
$$\int_{X}\zeta(\phi(v_k;x))\,\omega(x,v_k(x))\d\lambda(x) = \int_{\tilde C}\zeta(\phi(v_k;x))\,\omega(x,v_k(x))\d\lambda(x)$$
and, using  Jensen's inequality,
\begin{align*}
\int_{X}\!\!\zeta(\phi(v;x))\,\omega(x,v(x))\d\lambda(x)
&=\int_{\tilde C}\!\zeta(\phi(v;x))\,\omega(x,v(x))\d\lambda(x)\\
& \leq  \zeta\Big(\dfrac{\int_{\tilde C}\phi(v;x)\,\omega(x,v(x))\d\lambda(x)}{\int_{\tilde C}\omega(x,v(x))\d\lambda(x)}\Big)\int_{\tilde C}\,\omega(x,v(x))\d\lambda(x)\\
& \leq  \zeta\Big(\dfrac{\max_{x\in \tilde C}\omega(x,v(x))\Phi(v;\tilde C)}{\int_{\tilde C}\omega(x,v(x))\d\lambda(x)}\Big)\int_{\tilde C}\omega(x,v(x))\d\lambda(x).
\end{align*}
Using again 
that  $\Phi(v;\cdot)$ is a Radon measure, that $\tilde C$ is compact, and that $\omega(\cdot,v(\cdot))$ is a non-negative continuous function, we obtain that the integral is finite.
\end{proof}

We used the continuity of $v_k$ only to show that $\omega(\cdot, v_k(\cdot))$ is continuous. Hence, the proof of Theorem \ref{main} gives also the following result.

\begin{teo}\label{mainb}
For  $\zeta\in \conc$,
\begin{align*}
    \oZ(v)=\int_{X} \zeta(\phi(v;x))\d\lambda(x)
\end{align*}\
is finite for every $v\in \F$, and 
\begin{equation}
   \oZ(v)\ge \limsup_{k\to\infty} \oZ(v_k)
\end{equation}
for every sequence of functions $v_k$ in $\F$ with uniformly bounded supports $\supp (\Phi(v_k;\cdot))$
that converges to $v\in\F$.
\end{teo}

Let $\cvx$ be the set of decreasing convex functions $\zeta: [0,\infty)\rightarrow(0,\infty]$ such that  $\lim_{t\rightarrow 0}\zeta(t)=\zeta(0)$ and $\lim_{t\rightarrow \infty}\zeta(t)=0$. The following result is a functional version of \cite[Theorem 6]{Ludwig:gasa}.

%We allow that $\zeta(0)=\infty$ and therefore consider restrictions to compact sets in the following result, which is a functional version of   \cite[Theorem 6]{Ludwig:gasa}.

\goodbreak
\begin{teo}\label{lower}
Let   $\zeta\in \cvx$ and $\omega\in C(X\times\mathbb{R})$ be non-negative while $C\subseteq X$ is compact. Define $\oZ(\cdot, C)\colon\cF(X)\cap C(X)\to [0, \infty]$ by
\begin{align*}
    \oZ(v, \oom)=\int_{\oom} \zeta(\phi(v;x))\,\omega(x,v(x))\d\lambda(x).
\end{align*}
If $v\in\cF(X)\cap C(X)$, then
\begin{equation}\label{lsc}
    \liminf_{k\rightarrow \infty}\oZ(v_k, \oom)\geq \oZ(v,\oom)
\end{equation}
for every sequence $v_k\in\cF(X)\cap C(X)$ converging to $v\in \cF(X)\cap C(X)$.
\end{teo}

We allow that $\zeta(0) = \infty$. In this case, $\oZ(v, C) < \infty$ only if $C$ is contained in the support of $\Phi(v;\cdot)$, otherwise, the integrand attains the value $\infty$
%, and hence the functional becomes infinite. This is the reason why we consider restrictions to compact sets in  Theorem \ref{lower}.}

In the proof, we will also use the following version of Fatou's Lemma (see, for example, [\citealp{Rudin}, Lemma 1.28]),
which says that if $\mu_c$ is the counting measure, then
\begin{equation}\label{fatou}
    \sum_{y\in Y}\liminf_{k\rightarrow \infty}f_k(y)=\int_Y \left(\liminf_{k\rightarrow \infty}f_k(y)\right)\d\mu_c(y)\\ \leq %\liminf_{k\rightarrow \infty}\int_Y f_k(y)\d\mu_c(y)= 
    \liminf_{k\rightarrow \infty}\sum_{y\in Y} f_k(y)
\end{equation}
where $Y\subseteq \mathbb{Z}$ and $f_k: Y \to [0, \infty]$ for each $k\in \mathbb{Z}$.

\begin{proof}
Similar to the proof of Theorem \ref{main}, for $\oom\subset X$  compact, let $\oom_0\subset X$ be the $\lambda$-null set, where the singular part of $\Phi(v;\cdot)$ restricted to $\oom$ is concentrated. Choose by Lusin's theorem pairwise disjoint compact sets $\oom_l\subseteq \oom$ for $l\in\mathbb{N}$ such that $\phi(v;\cdot)$ is continuous as a function restricted to $\oom_l$ and such that for every $l\in \mathbb{N}$
\begin{align}\label{eq5.1}
    \oom_l\cap \oom_0=\emptyset
\end{align}
and such that 
\begin{align}\label{zero}
   \lambda\Big(\oom\setminus \bigcup\nolimits_{l=1}^{\infty}\oom_l\Big)=0.
\end{align}
We may assume that $\liminf_{k\to\infty} \oZ(v_k, \oom)<\infty$ and that $\oZ(v_k,\oom)<\infty$ for $k$ sufficiently large. Otherwise \eqref{lsc} clearly holds.

Let $v_k$ be a sequence on $\cF(X)\cap C(X)$ converging to $v\in\cF(X)\cap C(X)$. We will start by proving that for each $l\in\mathbb{N}$ the inequality
\begin{equation}\label{eq5.3}
   \liminf_{k\to\infty}\int_{\oom_l} \zeta(\phi(v_k;x))\,\omega(x,v_k(x))\d\lambda(x)\\
\geq \int_{\oom_l}\zeta(\phi(v;x))\,\omega(x,v(x))\d\lambda(x)
\end{equation}
holds. Let $\eta>0$. We choose an increasing monotone sequence $t_i\in (0,\infty)$ for $i\in\mathbb{Z}$ such that $ \lim_{i\rightarrow \infty}t_i=\sup\{ \zeta(t)\colon t\in [0, \infty))\}$ and $\lim_{i\rightarrow-\infty}t_i=0$ while

\begin{align}\label{eq5.4}
    \max\nolimits_{i\in\mathbb{Z}}|\zeta(t_{i+1})-\zeta(t_i)|\leq \eta
\end{align}
and such that for $i\in\mathbb{Z}$ and $k\geq 0$,
\begin{align}\label{eq5.5}
    \lambda(\{x\in \oom: \phi(v_k;x)=t_i\})=0,
\end{align}
where $v_0=v$. Set
\begin{align}\label{eq5.6}
    \oom_{li}=\{x\in\oom_l: t_i\leq \phi(v;x)\leq t_{i+1}\}.
\end{align}
Since $\phi(v;\cdot)$ is continuous on $\oom_l$ and $\oom_l$ is compact, the sets $\oom_{li}$ are compact for $i\in\mathbb{Z}$. Hence, $\omega(\cdot,v_k(\cdot))$ and $\omega(\cdot,v(\cdot))$ are uniformly continuous on $\oom_{li}$, and using \eqref{weak-star}  we get 
\begin{equation}\label{eq5.7}
   \limsup_{k\to\infty} \int_{\oom_{li}}\phi(v_k;x)\,\omega(x,v_k(x))\d\lambda(x)\leq \int_{\oom_{li}}\phi(v;x)\,\omega(x,v(x))\d\lambda(x).
\end{equation}
By \eqref{eq5.1} and the definition of $\oom_{li}$,
\begin{align}\label{eq5.8}
    \int_{\oom_{li}}\phi(v; x)\,\omega(x,v(x))\d\lambda(x)\leq t_{i+1}\int_{\oom_{li}}\omega(x,v(x))\d\lambda(x).
\end{align}

Since $\zeta$ is decreasing, we obtain
\begin{align}\label{eq5.9}
    \int_{\oom_l}\zeta(\phi(v; x))\,\omega(x,v(x))\d\lambda(x)\leq \sum_{i\in\mathbb{Z}} \zeta(t_i)\int_{\oom_{li}}\omega(x,v(x))\d\lambda(x).
\end{align}
Let $J_k\subset\mathbb{Z}$ be such that $\int_{\oom_{li}}\omega(x,v_k(x))\d\lambda(x)>0$ precisely if $i\in J_k$, and set $v_0=v$. If $i\in J_0$, then $i\in J_k$  for $k$ sufficiently large. By \eqref{eq5.5},  Jensen's inequality, and the convexity of $\zeta$, we obtain
\begin{equation}\label{eq5.10}
\begin{split}
    \int_{\oom_{l}}&\zeta(\phi(v_k; x))\,\omega(x,v_k(x))\d\lambda(x)\\
    &\geq \sum_{i\in J_k}\zeta\left(\dfrac{\int_{\oom_{li}}\phi(v_k; x)\,\omega(x,v_k(x))\d\lambda(x)}{\int_{\oom_{li}}\omega(x,v_k(x))\d\lambda(x)}\right)\int_{\oom_{li}}\omega(x,v_k(x))\d\lambda(x).
\end{split}    
\end{equation}
Using \eqref{eq5.10}, \eqref{fatou}, the monotonicity of $\zeta$, and the classical form of Fatou's lemma,   we obtain
\begin{align*}
    \liminf_{k\to\infty} \int_{\oom_l} &  \zeta(\phi(v_k; x))\,\omega(x,v_k(x))\d\lambda(x)\\
    &\hskip -22pt\geq \liminf_{k\to\infty}\sum_{i\in J_k}\zeta\left(\dfrac{\int_{\oom_{li}}\phi(v_k; x)\,\omega(x,v_k(x))\d\lambda(x)}{\int_{\oom_{li}}\omega(x,v_k(x))\d\lambda(x)}\right)\int_{\oom_{li}}\omega(x,v_k(x))\d\lambda(x)\\
    & \hskip -22pt\geq \sum_{i\in J_0}\liminf_{k\to\infty}\left(\zeta\left(\dfrac{\int_{\oom_{li}}\phi(v_k; x)\,\omega(x,v_k(x))\d\lambda(x)}{\int_{\oom_{li}}\omega(x,v_k(x))\d\lambda(x)}\right)\int_{\oom_{li}}\omega(x,v_k(x))\d\lambda(x)\right)\\ 
   % &\hskip -22pt \geq \sum_{i\in J_0} \zeta\Big(\limsup_{k\to\infty}\dfrac{\int_{\oom_{li}}\phi(v_k; x)\,\omega(x,v_k(x))\d\lambda(x)}{\int_{\oom_{li}}\omega(x,v_k(x))\d\lambda(x)}\Big)\liminf_{k\to\infty}\int_{\oom_{li}}\omega(x,v_k(x))\d\lambda(x)\\
 &\hskip -22pt \geq \sum_{i\in J_0} \zeta\Big(\dfrac{\limsup_{k\to\infty}\int_{\oom_{li}}\phi(v_k; x)\,\omega(x,v_k(x))\d\lambda(x)}{\int_{\oom_{li}}\omega(x,v(x))\d\lambda(x)}\Big)\int_{\oom_{li}}\omega(x,v(x))\d\lambda(x),
\end{align*}
and by $\eqref{eq5.7}, \eqref{eq5.8}, \eqref{eq5.9}$, and \eqref{eq5.4}, we conclude that
\begin{align*}
    \liminf_{k\to\infty} \int_{\oom_l} &  \zeta(\phi(v_k; x))\, \omega(x,v_k(x))\d\lambda(x)\\
    &\geq \sum_{i\in J_0} \zeta\left(\dfrac{\int_{\oom_{li}}\phi(v; x)\,\omega(x,v(x))\d\lambda(x)}{\int_{\oom_{li}}\omega(x,v(x))\d\lambda(x)}\right)\int_{\oom_{li}}\omega(x,v(x))\d\lambda(x)\\
    &\geq \sum_{i\in \mathbb{Z}}\zeta(t_{i+1})\int_{\oom_{li}}\omega(x,v(x))\d\lambda(x)\\
    &=\sum_{i\in \mathbb{Z}}(\zeta(t_{i})-(\zeta(t_{i})-\zeta(t_{i+1})))\int_{\oom_{li}}\omega(x,v(x))\d\lambda(x)\\
    & \geq \int_{\oom_{l}}\zeta(\phi(v;x))\,\omega(x,v(x))\d\lambda(x)-\eta\int_{\oom_l}\omega(x,v(x))\d\lambda(x).
\end{align*}
Since $\eta>0$ is arbitrarily, this proves \eqref{eq5.3}.
Finally, \eqref{zero}, \eqref{eq5.5}, \eqref{eq5.3}, and Fatou's Lemma imply
\begin{align*}
\!\liminf_{k\to\infty} \int_{\oom}\! \zeta(\phi(v_k;x))\,\omega(x,v_k(x))\d\lambda(x)
& = \liminf_{k\to\infty} \sum_{l=1}^{\infty}\int_{\oom_l} \!\zeta(\phi(v_k;x))\,\omega(x,v_k(x))\d\lambda(x)\\
& \geq \sum_{l=1}^{\infty}\liminf_{k\to\infty}\int_{\oom_l} \zeta(\phi(v_k;x))\,\omega(x,v_k(x))\d\lambda(x)\\
%& \geq \sum_{l=1}^{\infty} \int_{\oom_l} \zeta(\phi(v;x))\,\omega(x,v(x))\d\lambda(x)\\
& = \int_{\oom} \zeta(\phi(v;x))\,\omega(x,v(x))\d\lambda(x),
\end{align*}
as we wanted to prove.
\end{proof}

\goodbreak
We used the continuity of $v_k$ only to show that $\omega(\cdot,v_k(\cdot))$ is continuous. Hence,
we also have the following version of Theorem \ref{lower}.

\begin{teo}\label{lower2}
Let   $\zeta\in \cvx$ and     $C\subseteq X$ compact. For $v\in \cF(X)$, define
 $  \oZ(v,C)=\int_{C} \zeta(\phi(v;x))\d\lambda(x)$ (as an element of $[0, \infty]$).
If $v\in \cF(X)$, then
\begin{align*}
    \liminf_{k\rightarrow \infty}\oZ(v_k,C)\geq \oZ(v,C)
\end{align*}
for every sequence $v_k\in \cF(X)$ that converges to $v$. %with uniformly bounded supports of $\Phi(v_k;\cdot)$ 
\end{teo}

\section{Proof of Theorem \ref{teo1p} and Further Applications}\label{pthm}

%Using \eqref{conjugate} and Lemma \ref{duality}, we obtain Theorem \ref{teo1p} from the following dual version, which is a special case of Theorem \ref{mainb}.
We obtain the following result as special case of Theorem \ref{mainb}.

\begin{teo}\label{coro}
For $\zeta\in \conc$ and $v\in  \Conv_{\MA}\finite$, 
\begin{align*}
    \oZ(v)=\int_{\mathbb{R}^n} \zeta(\det(\Hess v(x)))\d x
\end{align*} 
is finite, and
\begin{equation}
    \oZ(v)\ge \limsup_{k\to \infty} \oZ(v_k)
\end{equation}
for every sequence  $v_k$ in $\Conv_{\MA}\finite$ that $\tau^*$-converges to $v$.
\end{teo}

A version of Theorem \ref{coro} also applies to 
$\Conv_{\MAn}(\mathbb{R}^n;\mathbb{R})$, which implies by  \eqref{conjugate} and Lemma \ref{duality} that Theorem \ref{teo1p} holds. 
We apply Theorem \ref{mainb} also to the weighted functional affine surface area from~\cite{STTW} and remark that more general weighted affine surface areas \cite{LSW,STTW} are also covered by Theorem~\ref{mainb} on $\Conv_{\MA}\finite$.

\begin{teo}\label{corow}
The functional $\oZ\colon\Conv_{\MA}\finite\to [0, \infty)$, defined by
\begin{align*}
    \oZ(v)=\int_{\mathbb{R}^n} \det(\Hess v(x))^{\frac{1}{n+2}} e^{-\frac n{n+2}v(x)}\d x,
\end{align*} 
is $\tau^*$-upper semicontinuous.
\end{teo}

The following result is a special case of Theorem~\ref{lower2}.

\begin{teo}\label{coro3}
Let   $\zeta\in \cvx$ and $C\subset \R^n$ compact.   Define 
$$
   \oZ(v, C)=\int_{C}\zeta(\det(\Hess v(x)))\d x$$
for $v\in\Conv_{\MA}\finite$
(as an element of $[0, \infty]$).
If $v\in\Conv_{\MA}\finite$,  then
\begin{equation}
    \oZ(v,C)\le \liminf_{k\to \infty} \oZ(v_k,C)
\end{equation}
for every sequence  $v_k$ in $\Conv_{\MA}\finite$ that $\tau^*$-converges to $v$.
\end{teo}

We remark that instead of Monge--Ampère measures mixed Monge--Ampère measures and, in particular, Hessian measures can be used here. Such measures also appear in the characterization of continuous valuations on convex functions \cite{Colesanti-Ludwig-Mussnig-7}.

\goodbreak

\subsection*{Acknowledgments}
The authors would like to thank Fabian Mussnig for his careful reading of the
manuscript and  fruitful discussions.
This research was funded in part by the Austrian Science Fund (FWF) Grant-DOI $\colon$ 10.55776/P37030. 
For open access purposes, the authors have applied a CC BY public copyright license to any author accepted manuscript version arising from this submission.

\bibliography{usc}
\bibliographystyle{amsplain}
\end{document}